\documentclass[12pt,a4paper]{amsart}

\usepackage[a4paper, total={6in, 23.7cm}]{geometry}
\usepackage{amsmath,amsthm,amsfonts,amssymb,mathrsfs,amsbsy}
\usepackage{graphicx}
\usepackage{color}
\usepackage{cite}
\usepackage{marginnote}
\usepackage{xifthen}
\usepackage{mathtools}
\usepackage{dsfont}
\usepackage{bbm}
\usepackage{appendix}
\usepackage{delarray}
\usepackage{enumerate}

\usepackage{xcolor}
\usepackage{tikz}
\usepackage[hypertexnames=false]{hyperref} 
\usepackage{autonum}

\binoppenalty=\maxdimen
\relpenalty=\maxdimen

\newenvironment{spm}
	{\bigl(\begin{smallmatrix}}
	{\end{smallmatrix}\bigr)}

\newenvironment{nalign}
	{\begin{equation}\begin{aligned}}
	{\end{aligned}\end{equation}\ignorespacesafterend}

\newtheorem{theorem}{Theorem}

\newtheorem{defi}[theorem]{Definition}

\newtheorem{lemma}[theorem]{Lemma}
\newtheorem{prop}[theorem]{Proposition}

\numberwithin{equation}{section}
\numberwithin{theorem}{section}
\numberwithin{figure}{section}

\theoremstyle{remark}
\newtheorem{rem}[theorem]{Remark}

\theoremstyle{remark}

\theoremstyle{remark}

\newcommand\R{\mathbb{R}}
\newcommand\C{\mathbb{C}}

\newcommand\N{\mathcal{N}}

\newcommand\bu{\pmb u}
\newcommand\bff{\pmb f}
\newcommand\bA{\pmb A}
\newcommand\bB{\pmb B}
\newcommand\bC{\pmb C}

\newcommand\dx{\,\text{d}x}

\newcommand{\UV}{(\pmb U, V)}
\newcommand{\UVx}{\big(\pmb U(x), V(x)\big)}

\newcommand\V{\mathcal{V}}

\newcommand\bh{\bar h}

\renewcommand\L{\mathcal{L}}
\renewcommand\Re{{\rm Re \,}}

\newcommand{\retainlabel}[1]{\label{#1}\sbox0{\ref{#1}}}

\newcommand{\W}[2]
{{
	W^{#1,#2}_
	{
		\ifthenelse
		{
			\equal{#1}{2}
		}
		{
			\nu
		}
		{
		}
	}(\Omega)
}}

\newcommand{\Li}{{L^\infty(\Omega)}}
\newcommand{\Lp}[1]
{{
	L^
	{
		\ifthenelse
		{
			\isempty{#1}
		}
		{
			p
		}
		{
			#1
		}
	}(\Omega)
}}

\definecolor{lime}{HTML}{A6CE39}
\DeclareRobustCommand{\orcidicon}{%
	\begin{tikzpicture}
		\draw[lime, fill=lime] (0,0) 
		circle [radius=0.16] 
		node[white] {{\fontfamily{qag}\selectfont \tiny ID}};
		\draw[white, fill=white] (-0.0625,0.095) 
		circle [radius=0.007];
	\end{tikzpicture}
	\hspace{-2mm}
}

\begin{document} 

\title[Reaction-diffusion-ODE systems]{
Instability of all regular stationary solutions \\ to reaction-diffusion-ODE systems}

\author[S. Cygan]{Szymon Cygan \href{https://orcid.org/0000-0002-8601-829X}{\orcidicon}}
\address[S. Cygan]{	Instytut Matematyczny, Uniwersytet Wroc\l{}awski, pl. Grunwaldzki 2/4, \hbox{50-384} Wroc\l{}aw, Poland \\ 
	\href{https://orcid.org/0000-0002-8601-829X}{orcid.org/0000-0002-8601-829X}}
\email{szymon.cygan@math.uni.wroc.pl}
\urladdr {http://www.math.uni.wroc.pl/~scygan}

\author[A. Marciniak-Czochra]{Anna Marciniak-Czochra
	\href{https://orcid.org/0000-0002-5831-6505}{\orcidicon}}
\address[A. Marciniak-Czochra]{
	Institute of Applied Mathematics,  Interdisciplinary Center for Scientific Computing (IWR) and BIOQUANT, University of Heidelberg, 69120 Heidelberg, Germany\\
	\href{https://orcid.org/0000-0002-5831-6505}{orcid.org/0000-0002-5831-6505}}
\email{anna.marciniak@iwr.uni-heidelberg.de}
\urladdr {http://www.biostruct.uni-hd.de/}

\author[G. Karch]{Grzegorz Karch \href{https://orcid.org/0000-0001-9390-5578}{\orcidicon}}
\address[G. Karch]{	
	Instytut Matematyczny, Uniwersytet Wroc\l{}awski, pl. Grunwaldzki 2/4, \hbox{50-384} Wroc\l{}aw, Poland \\ 
	\href{https://orcid.org/0000-0001-9390-5578}{orcid.org/0000-0001-9390-5578}}
\email{grzegorz.karch@math.uni.wroc.pl}
\urladdr {http://www.math.uni.wroc.pl/~karch}

\author[K. Suzuki]{Kanako Suzuki\href{https://orcid.org/0000-0001-8018-9621}{\orcidicon}}
\address[K. Suzuki]{
	Graduate School of Science and Engineering, Ibaraki University,
	2-1-1 Bunkyo, Mito 310-8512, Japan\\
	\href{https://orcid.org/0000-0001-8018-9621}{orcid.org/0000-0001-8018-9621}
}
\email{kanako.suzuki.sci2@vc.ibaraki.ac.jp}

\date{\today}

\thanks{The authors appreciated helpful comments on this work by Dr.~Chris Kowall.
S.~Cygan acknowledges a  support  by the Polish NCN grant 2016/23/B/ST1/00434.  This work was also supported by the Deutsche Forschungsgemeinschaft (DFG, German Research Foundation) under Germany's Excellence Strategy EXC 2181/1 - 390900948 (the Heidelberg STRUCTURES Excellence Cluster and SFB1324 (B05), and JSPS the Grant-in-Aid for Scientific Research (C) 18K03354 and 19K03557.}


\begin{abstract}
	A general system of several ordinary differential equations coupled with a reaction-diffusion equation in a bounded domain with zero-flux boundary condition is studied in the context of pattern formation. These initial-boundary value problems may have regular ({\it i.e.}~sufficiently smooth) stationary solutions. This class of  {\it close-to-equilibrium} patterns includes stationary solutions that emerge due to the Turing instability of a spatially constant stationary solution. The main result of this work is instability of all regular patterns. It suggests that stable stationary solutions arising in models with non-diffusive components must be {\it far-from-equilibrium} exhibiting singularities. Such discontinuous stationary solutions have been considered in our parallel work [\textit{Stable discontinuous stationary solutions to reaction-diffusion-ODE systems}, preprint (2021)].
\end{abstract}


\maketitle

\section{Introduction} 

We study  properties of solutions of a  general system of $n$ ordinary differential equations coupled with a scalar reaction-diffusion equation 
\begin{nalign}
	\bu_t  &=   \bff(\bu,v), &&x\in\overline{\Omega}, \quad t>0,  \label{eq1}\\
	v_t  &=   \Delta v+g(\bu,v), &&  x\in \Omega, \quad t>0,
\end{nalign}
with an unknown vector field (denoted using a bold-face font) and a scalar function: 
\begin{nalign}
	\bu = \bu (x,t) = 
	\left( 
	\begin{array}{c}
	u_1(x,t) \\ 
	\vdots \\
	u_n(x,t)
	\end{array}
	\right) 
	\qquad\text{and}\qquad 
	v = v(x,t).
\end{nalign}
We consider arbitrary $C^2$-nonlinearities 
\begin{nalign} 
	\label{nonlinearity}
	\bff=\bff(\bu,v) = 
	\left(
	\begin{array}{c}
	f_1(\bu,v) \\
	\vdots\\
	f_n(\bu,v)
	\end{array}
	\right)
	\qquad\text{and}\qquad  
	g=g(\bu,v)
\end{nalign}
and define system \eqref{eq1}  in  a bounded open domain $\Omega\subset \R^N$ for $N \geq 1$ with $C^2$-boundary $\partial\Omega$.
Moreover,   the reaction-diffusion equation in \eqref{eq1} is 
supplemented by the homogeneous Neumann boundary condition \noeqref{Neumann}
\begin{nalign}
	\retainlabel{Neumann}
	\partial_{\nu}v  =  0  \qquad \text{for}\quad x\in \partial\Omega, \quad t>0,
\end{nalign}
where $\partial_\nu =  \nu\cdot\nabla$ with  $\nu$ denoting  the unit
outer normal vector to $\partial \Omega$. We also impose an initial condition \noeqref{ini}
\begin{nalign}
	\label{ini}
	&&\bu(x,0)  =   \bu_{0}(x)=
	\left( 
	\begin{array}{c}
	u_{0,1}(x) \\ 
	\vdots \\
	u_{0,n}(x)
	\end{array}
	\right) 
	,\qquad v(x,0)  =  v_{0}(x).
\end{nalign}

It is well-known that problem \eqref{eq1}-\eqref{ini} has a unique local-in-time solution corresponding to a continuous initial condition, see \textit{e.g.} \cite{MR3039206,MR3600397}. Our goal is to study stability of stationary solutions $\UV=\UVx$ satisfying the relation 
\begin{nalign}
	\label{eq:ImpFun}
	\pmb f \UVx  = 0, \quad x\in \Omega
\end{nalign}
and the boundary value problem 
\begin{nalign}
	\label{eq:ImFun2}
	&\Delta V + g\UV = 0, && x\in \Omega, \\
	&\partial_{\nu} V = 0, && x\in \partial\Omega.
\end{nalign}
Problem \eqref{eq:ImpFun}-\eqref{eq:ImFun2} may have different types of solutions. Here, we generalize results from works \cite{MR3973251, MR3679890, MR3600397, MR3583499, MR3345329, MR3214197, MR3059757, MR3039206, MR2833346,MR684081,MR579554,10780947202173109,MR4213664,MR730252} where the following two types of stationary solutions were considered.

\begin{enumerate}
	\item Stationary solutions of the first type are called \textit{regular}. In this case, the function  $\pmb U(x)$ is obtained from $V(x)$ by solving equation \eqref{eq:ImpFun} in such a way that $\pmb U(x) = \pmb k\big(V(x)\big)$ for some $C^2$-function $\pmb k=(k_1, ...,k_n)$.  \label{it:Type1}
	\item The second class of stationary solutions is called {\textit{jump-discontinuous}}, because $\pmb U(x)$ is obtained from $V(x)$ in a discontinuous way by choosing different branches of solutions to equation \eqref{eq:ImpFun}. \label{it:Type2} 
\end{enumerate} 
In this work, we deal with the first class of solutions and we show that all regular stationary solutions  are unstable -- see Theorems \ref{thm:Aposit} and \ref{thm:non-const} below for precise statements of our results. In our second work \cite{CMCKS02}, we study discontinuous stationary solutions to problem \eqref{eq1}-\eqref{ini} 
and we find sufficient conditions for their existence and stability. Other types of stationary solutions, different from those mentioned above, may also exist and we comment them in Remark \ref{rem:Other} below.

It is well known that a single autonomous reaction diffusion equation $v_t = \Delta v + h(v)$ in a bounded convex domain and with the Neumann boundary condition does not support interesting patterns and only constant solutions can be stable, see Casten-Holland \cite{MR480282} and Matano \cite{MR555661}. Similarly, it was shown by Kishimoto and Weinberger \cite{MR791838} that there is no stable non-constant equilibrium solutions of homogeneous cooperative-diffusion systems in convex domains with no-flux boundary conditions. Result from \cite{MR791838} were recently extended by Wang \cite{MR3921221} on reaction-diffusion-ODE cooperative systems. 

In this paper, we contribute to the theory by showing that all {\it close-to-equilibrium} (regular) patterns of  a general reaction-diffusion-ODE problem \eqref{eq1}-\eqref{ini} are unstable, regardless of the particular structure assumption on nonlinearities. In particular, it implies that such models cannot exhibit stable Turing patterns and the only possible stable stationary solutions of such models have to be somewhat  singular or discontinuous, except a certain marginal class of degenerate models discussed below in Remark \ref{rem:Other}. In this way, we generalize a series of previous works where either an existence or stability of either regular or discontinuous  stationary solutions have been analysed  in the case of specific reaction-diffusion-ODEs from mathematical biology  \cite{MR3039206,MR3600397,Kthe2020HysteresisdrivenPF,MR3973251,MR3583499}. Finally, we refer the reader to our work \cite{CMCKS02} for a discussion of discontinuous stationary solutions for some particular models considered \textit{e.g.}~in papers \cite{10780947202173109,MR3214197,MR3583499,Kthe2020HysteresisdrivenPF,MR2205561,MR3679890, MR3973251, MR3039206, MR3600397,MR3345329,MR4213664}. Related results on stationary solutions of reaction-diffusion-ODE systems can be also found in \cite{MR2297947, MR2527521,MR1036472,MR730252,MR684081,MR579554,perthame2020fast}.

In Section \ref{sec:MainResults}, we present and discuss main results of this work. In Proposition \ref{thm:reg}, we construct a non-constant family of regular stationary solutions to problem~\eqref{eq1}-\eqref{ini} using the bifurcation theory. The proof of that proposition is postponed to Section~\ref{sec:RegularSolution}. Instability results for regular stationary solutions are stated in Theorem~\ref{thm:Aposit} and Theorem~\ref{thm:non-const} whose proofs are contained in Section \ref{sec:InstRes}. In Section \ref{sec:LinearEquation}, we characterize spectra of linearised operators corresponding to problem~\eqref{eq1}-\eqref{ini}.

\subsection*{Notation.} 

By the bold font, e.g.\ $\pmb A, \pmb u,$ we denote either matrices or vector valued functions in order to distinguish them from scalar quantities.
For $\pmb f$ and $g$ defined in~\eqref{nonlinearity} we set 
\begin{nalign}
	\pmb f_{\pmb u} &= \begin{pmatrix}
		\frac{\partial f_1}{\partial u_1}  & \cdots & \frac{\partial f_1}{\partial u_n}  \\
		\vdots & & \vdots \\
		\frac{\partial f_n}{\partial u_1}  & \cdots & \frac{\partial f_n}{\partial u_n}  \\
	\end{pmatrix}, & 	
	\pmb f_{v} &= 
	\begin{pmatrix}
		\frac{\partial f_1}{\partial v}  \\
		\vdots  \\
		\frac{\partial f_n}{\partial v} \\
	\end{pmatrix}, \;  
	g_{\pmb u} &= 
	\begin{pmatrix}
		\frac{\partial g}{\partial u_1}  & \cdots & \frac{\partial g}{\partial u_n}  \\
	\end{pmatrix}, & 	\; 
	g_{v} &= \frac{\partial g}{\partial v} .
\end{nalign} 
$Y^n = {Y\times \cdots \times Y}$ ($n$-times) is the product of a given space $Y$ with the norm denoted by $\|\pmb  y\|_Y$  (instead of $\| \pmb y \|_{Y^n}$). The symbol  $\sigma(L)$ means the spectrum of a linear operator $\big( L, D(L)\big)$ and $s(L) = \sup\lbrace \Re \lambda : \lambda \in \sigma(L) \rbrace$ is the  spectral bound of an operator $\big( L, D(L)\big)$. The usual Sobolev space $W^{1,2}(\Omega)$ has the scalar product 
\begin{nalign}
	\langle u,v\rangle_{1,2}=\int_\Omega \nabla u\cdot \nabla v\,dx+\int_\Omega uv\dx.
\end{nalign}
We use the symbol  $\Delta_\nu$ for the Laplace operator in a bounded domain $\Omega$ with the Neumann boundary condition. It is defined in the usual way via a bilinear form on $W^{1,2}(\Omega)$, and $-\Delta_\nu$ has the eigenvalues $\mu_k$  satisfying $0 = \mu_0 < \mu_1 \leqslant \cdots \leqslant \mu_n \to \infty$.  Constants in the estimates below are denoted by the same letter $C$, even if they vary from line to line. Sometimes we shall emphasize dependence of such constants on parameters used in our calculations.

\section{Main results} 
\label{sec:MainResults}

The goal of this work is to construct  certain non-constant stationary solutions of problem \eqref{eq1}-\eqref{ini} and to study their stability. Thus, we deal with a solution $(\pmb U, V) = \big( \pmb U(x), V(x) \big)$ to the boundary value problem
\begin{nalign}
	\pmb f(\pmb U,V)&=0,  &&     &&x\in\overline{\Omega},  && \label{seq1}\\
	\Delta_\nu V+g(\pmb U,V)&=0,  &&  &&x\in \Omega,
\end{nalign}
with arbitrary $C^2$-functions $\pmb f$ and $g$ of the form \eqref{nonlinearity}  and in a bounded domain  
$\Omega\subset \R^N $ with  $C^2$-boundary.

\begin{defi}\label{def:stat}
	A pair $\UV = \UVx$ 
	is 
	a {\it weak solution} of problem   \eqref{seq1} 
	if 
	\begin{itemize}
		\item$\pmb U$ is measurable, 
		\item $V\in \W12$,
		\item $g(\pmb U, V) \in \big(\W12\big)^\star$ $($the dual of the space $\W12)$,    
		\item the equation $\pmb f\big(\pmb U(x),V(x)\big)=0$ is satisfied for almost all $x\in \Omega$,
		\item the equality 
			\begin{nalign}
				-\int_\Omega \nabla V(x)\cdot \nabla \varphi(x)\dx +\int_\Omega 		g\big(\pmb U(x),V(x)\big)\varphi(x)\dx=0 
			\end{nalign}
	holds true for all test functions $\varphi\in \W12$.
	\end{itemize}
\end{defi}

\begin{defi}\label{def:reg}
	A weak solution in the sense of Definition \ref{def:stat} to problem \eqref{seq1} is called regular if $\pmb U, V \in \Li$ and, moreover, 
	there exists  a $C^2$-function $\pmb k: \R\to \R^n$ such that $\pmb U(x)=\pmb k(V(x))$ for all $x\in \Omega$.
\end{defi}

\begin{rem}
	\label{thm:RegSolProp}
	Notice that every  regular solution of problem \eqref{seq1} satisfies the equation
	\begin{nalign}
		\pmb f\big(\pmb U(x),V(x)\big) = \pmb f\big(\pmb k(V(x)),V(x)\big)=0 \quad \text{for all}\quad x\in \Omega,
	\end{nalign}
	where $V=V(x)$ is a solution of  the elliptic Neumann problem  
	\begin{align}
		\Delta_\nu V+h(V)=0\qquad   \text{for}\quad x\in \Omega \label{s:h} 
	\end{align}
	with $h(V)=g\big(\pmb k(V),V\big)$. 
	Since we require $V\in \Li$ and since $\partial \Omega$ is $C^2$, by a standard elliptic regularity, we have $V\in W^{2,p}(\Omega)$ for each $p\in (1,\, \infty)$ and consequently $\pmb U \in W^{2,p}(\Omega)^{n}$ for every $p\in (1,\infty)$. In particular, every regular stationary solution satisfies $\UV \in C(\overline{\Omega})^{n+1}$.
\end{rem}
\begin{rem}
	Assume that $\UVx$ is a non-constant regular solution to problem~\eqref{def:stat}. Then, there exists $x_0 \in \Omega$ such that a vector $(\overline{\pmb U}, \overline{V}) = \big(\pmb U(x_0), V(x_0)\big)$ is a constant solution to this problem. To prove this fact, it suffices to integrate the second equation in \eqref{def:stat} over $\Omega$ and to use Neumann boundary condition to obtain $\int_{\Omega} g\big(\pmb U(x), V(x)\big) d x=0$. Since $\pmb U$ and $V$ are continuous there exists $x_0$ such that 
	$g\big(\pmb U(x_0), V(x_0)\big) =0$ and by the first equation in \eqref{def:stat} we obtain $\pmb f\big(\pmb U(x_0), V(x_0)\big)=0$.
\end{rem}

Solutions to the general Neumann elliptic problem \eqref{s:h}  has been constructed in several works. For example, classical bifurcation methods has been used in the works \cite{MR808736, MR727393,MR3679890,MR3973251} and the phase portrait method in \cite{MR3039206}. We refer the reader also to the review by Ni \cite{MR2103689} and to the references therein. Here, we recall one possible construction of solutions to problem \eqref{s:h} using a variational approach to bifurcation methods. 

\begin{prop}\label{thm:reg}
	Let $N \leqslant 6$. Consider a constant stationary solution of problem~\eqref{seq1}, namely, the constant vector $\big(\overline{\pmb U},\overline{V}\big)\in \R^{n+1}$ such that 
	\begin{nalign}
		\pmb f(\overline{\pmb U},\overline{V})=0\quad \text{and} \quad  g(\overline{\pmb U},\overline{ V})=0.
	\end{nalign}
	Define the following matrices
	\begin{nalign}
	\label{matrices}
	\pmb A_0 = D_{\pmb u} \pmb f (\overline{\pmb U},\overline{ V}), \quad  \pmb B_0 = D_{v} \pmb f (\overline{\pmb U},\overline{ V}), \quad 
	\pmb C_0 = D_{\pmb u} g (\overline{\pmb U},\overline{ V}), \quad
	d_0 = D_{v}  g (\overline{\pmb U},\overline{ V}) 
	\end{nalign}
	and assume that
	\begin{nalign}
		\det \bA_0 \neq 0 \quad\text{and} \quad \frac{1}{\det \bA_0 }\det\begin{pmatrix}
		\bA_0  & \bB_0  \\
		\bC_0  & d_0
		\end{pmatrix} =\mu_k> 0,
	\end{nalign}
	where $\mu_k$ is one of the eigenvalues of $-\Delta_\nu$. Then, there exists a sequence of real numbers  $d_\ell\to 1$ such that  the following ``perturbed'' problem
	\begin{nalign}\label{seq1:d}
	\pmb f(\pmb U,V)=0,&&     \quad &x\in\overline{\Omega},  && \\
	d_\ell \Delta_\nu V+(1-d_\ell)(V-\overline{ V})+g(\pmb U,V)=0,&&  \quad &x\in \Omega
	\end{nalign}
	has a non-constant regular solution. 
\end{prop}

\begin{rem}
	\label{rem:LinearExistence}
	Let us illustrate Proposition \ref{thm:reg} in the case of the linear problem 
		\begin{nalign}
		\label{eq:LinCon}
		   \bA_0  \pmb U+\bB_0 V &=0 , &     \quad &x\in\overline{\Omega},  &&
		\\     \Delta_\nu V+\bC_0   \pmb U+d_0 V &= 0,&  \quad &x\in \Omega,
	\end{nalign}
	with arbitrary constant coefficient matrices (not necessary as those in \eqref{matrices}) 
	\begin{nalign}
		\label{eq:LinConMat}
		\bA_0 &= 
		\begin{pmatrix}
			a_{11}&\dots&a_{1n}\\
			\vdots&\ddots&\vdots\\
			a_{n1}&\dots&a_{nn}
		\end{pmatrix},
		&\bB_0&= 
		\begin{pmatrix}
			b_{1}\\
			\vdots\\
			b_{n}\\
		\end{pmatrix},
		\\
		\bC_0 &= 
		\begin{pmatrix}
			c_{1}&\dots&c_{n}
		\end{pmatrix}, 
		&d_0 &= d.
	\end{nalign}
	Under the assumption
	\begin{nalign}
		\label{eq:LinearNonzeroStationarySolutions}
		\det \bA_0 \neq 0 \quad \text{and} \quad
		-\pmb C_0 \pmb A^{-1}_0 \pmb  B_0 + d_0 =\mu_k> 0
	\end{nalign}
	problem \eqref{eq:LinCon} has non-constant regular stationary solutions of the form
	\begin{nalign}
		\label{eq:LinStatSol}
	\begin{pmatrix} \pmb U\\V\end{pmatrix} =\begin{pmatrix}-\bA_0 ^{-1}\bB_0  \Phi_k\\ \Phi_k\end{pmatrix}, 
	\end{nalign}
	where  $\Phi_k$  is an eigenfunction of $-\Delta_\nu$ corresponding to the eigenvalue $\mu_k$. By relation~\eqref{eq:DetIden} below, assumptions \eqref{eq:LinearNonzeroStationarySolutions} take the form \begin{nalign}
		\label{eq:MatCon}
		\det \bA_0 \neq 0 \quad \text{and} \quad
		\frac{1}{\det \bA_0 }\det\begin{pmatrix}
		\bA_0  & \bB_0  \\
		\bC_0  & d_0 \end{pmatrix} = \mu_k >0.
	\end{nalign}
	Proposition \ref{thm:reg} shows how to extend this construction to nonlinear systems by using a bifurcation argument.
\end{rem}

In order to formulate our instability results, we associate with a matrix  
\begin{nalign}
	\bA &= \pmb A(x) =
	\begin{pmatrix}
		a_{11}(x)&\dots&a_{1n}(x)\\
		\vdots&\ddots&\vdots\\
		a_{n1}(x)&\dots&a_{nn}(x)
	\end{pmatrix},
\end{nalign}
where
\begin{nalign}
	\label{Coeffabcd}
	a_{ij} \,\in L^\infty(\Omega), \quad \text{for} \; i,j \in \lbrace 1, \cdots, n \rbrace,
\end{nalign}
the corresponding multiplication operator on $\Lp{}^n$, with $p \in [1, \, \infty)$, given by
\begin{nalign}
	\pmb \varphi \in \Lp{}^n  \mapsto \pmb A \pmb \varphi \in \Lp{}^n . 
\end{nalign}
We denote by $\sigma\big(\pmb A(\cdot) \big)$ the spectrum of this mapping and by $s\big(\pmb A(\cdot) \big)$ the spectral bound. Below in Lemma \ref{thm:SpectraOfMultiplicationByPartialyContinousFunction}, we show that in the case of $a_{i,j} \in C(\overline{\Omega})$ we have
\begin{nalign}
	s\big(\pmb A(\cdot) \big) = \sup \Big\lbrace \Re \lambda(x): \, \text{where }  \lambda(x) \text{ is an eigenvalue of } \pmb A(x) \text{ for some } x\in \overline{\Omega} \Big\rbrace
\end{nalign}

We are in a position to state main results of this work. We prove in Theorems~\ref{thm:Aposit} and  \ref{thm:non-const} below that all (\textit{i.e.}~not only those from Proposition \ref{thm:reg})  regular stationary solutions to problem \eqref{eq1}-\eqref{ini} are unstable except one degenerate case (see assumption~\eqref{eq:ConZeroDet}) which we discuss below in Remark~\ref{rem:Other}.  

\begin{theorem}\label{thm:Aposit}
	Let $N\geqslant 1$ and $\Omega \subset \R^N$ be bounded and open with a smooth boundary. Let $(\pmb U, V)$ be a weak regular stationary solution to problem \eqref{eq1}-\eqref{ini}. If
	\begin{nalign}
		\label{eq:AutoCatCon}
		s\big( f_{\pmb u} (\pmb U(\cdot), \, V(\cdot) )\big)> 0
	\end{nalign}
 	then $(\pmb U, V)$ is unstable in $C(\overline{\Omega})^{n+1}$. 
\end{theorem}

\begin{theorem}\label{thm:non-const}
	Let $N\geqslant 1$ and $\Omega \subset \R^N$ be bounded, open, convex and with a smooth boundary. Let $(\pmb U, V) = \big(\pmb U(x), V(x) \big)$ be a non-constant regular stationary solution of problem \eqref{eq1}-\eqref{ini} such that 
	\begin{nalign}\label{eq:sLeZero}
		s\big( f_{\pmb u} (\pmb U(\cdot), \, V(\cdot) )\big)\leqslant 0
	\end{nalign}	
	and 
	\begin{nalign}
		\label{eq:ConZeroDet}
		\det f_{\pmb u} \big(\pmb U(x), \, V(x) \big) \neq 0 \quad \text{for all} \quad x\in \overline{\Omega}. 
	\end{nalign}
	Then $(\pmb U, V)$ is unstable in $C(\overline{\Omega})^{n+1}$. 
\end{theorem}

\begin{rem}
	Let us illustrate both instability Theorems \ref{thm:Aposit} and \ref{thm:non-const} by applying them to the constant coefficients linear problem
	\begin{nalign}
		\bu_t  &=   \bA_0  \bu+\bB_0 v, &     \quad &x\in\overline{\Omega}, \quad t>0, &&
		\\  v_t  &=   \Delta_\nu v+\bC_0   \bu+d_0 v,&  \quad &x\in \Omega, \quad t>0,
	\end{nalign}
	which under condition \eqref{eq:MatCon} has a non-constant stationary solution \eqref{eq:LinStatSol}. This solution is unstable. Indeed, the second expression in \eqref{eq:MatCon} is the quotient of $n+1$ eigenvalues of the matrix $\begin{spm} \bA_0 & \bB_0 \\ \bC_0 & d_0\end{spm}$ and of $n$ eigenvalues of the matrix $\bA_0$. Obviously, all these eigenvalues are either real or pairwise conjugate. Thus, the second inequality in \eqref{eq:MatCon} implies that there exists at least one real eigenvalue $\lambda>0$. 
	If $\lambda$ is an eigenvalue of matrix $\bA_0$ then solution \eqref{eq:LinStatSol} is unstable by Theorem \ref{thm:Aposit}. Otherwise, $\lambda$ has to be an eigenvalue of matrix $\begin{spm} \pmb A_0 & \pmb B_0 \\ \pmb  C_0 & d_0 \end{spm}$ and solution \eqref{eq:LinStatSol} is unstable by Theorem~\ref{thm:non-const}, provided $\Omega$ is convex. In fact, this second instability result can be also obtained  directly from the fact that zero is an unstable solution of system \eqref{eq:LinCon} without the diffusion and the convexity of the domain is not needed.
\end{rem}

\begin{rem}
	\label{rem:Other}
	Condition \eqref{eq:ConZeroDet} has been  imposed  because 	some particular versions of  problem \eqref{eq1}-\eqref{ini} may have smooth stationary  solutions which are either  stable or unstable if $\det f_{\pmb u} \big(\pmb U(x), \, V(x) \big) = 0$ for some $x\in\overline\Omega$. As the simplest example, for $\pmb f \equiv 0$,  $g \equiv 0$, this problem has a solution of the form $\UV = (\pmb U, C)$ where $\pmb U=\pmb U(x)$ is an arbitrary smooth vector field  and $C\in \R$ is a constant. Obviously this is a stable solution, but not regular in the sense of Definition~\ref{def:reg}.	As another  example, we consider the problem 
	\begin{nalign}
		& u_t = 0,&& x\in \overline\Omega,\quad  t>0, \\
		&v_t = \Delta_\nu  v + (\mu_{1} +1)  u - v,&& x\in \Omega,\quad  t>0,
	\end{nalign}
	which has the regular stationary solution~$u=v=\Phi_1$ (recall that $\mu_1>0$ is the eigenvalue of $-\Delta_\nu$ with the eigenfunction $\Phi_1$) which is stable but not not asymptotically stable. On the other hand, the following modified problem 
\begin{nalign}\nonumber
		& u_t = 0,&& x\in \overline\Omega,\quad  t>0, \\
		&v_t = \Delta_\nu  v + (\mu_{1} -1)  u + v,&& x\in \Omega,\quad  t>0,
	\end{nalign}
	has the same regular stationary solution~$u=v=\Phi_1$  which is now unstable.
\end{rem}

\begin{rem}
	For one ODE coupled with one PDE, assumption \eqref{eq:AutoCatCon} reduces to the following inequality $f_{\pmb u} \big(\pmb U(x_0), \, V(x_0) \big)> 0$,  for some  
	$x_0 \in \overline{\Omega}$, which  is called  as an \textit{autocatalysis condition} in the work \cite{MR3600397}. In this case, an instability of all regular stationary solutions has been shown in \cite[Thm. 2.1]{MR3600397}.
\end{rem}


\section{Existence of regular stationary solutions} 
\label{sec:RegularSolution}

First, we prove a general result on the existence of non-constant solutions to Neumann boundary value problem \eqref{s:h}. This is rather  standard result obtained for example in \cite[Sec. 3]{MR808736} or \cite[Sec. 2]{MR727393} in the case of simple eigenvalues. We recall one possible construction of such solutions using variational methods in the following lemma which is a generalization of \cite[ Theorem 11.32]{MR845785}.

\begin {lemma}\label{lem:Rab}
	Let $N\leq 6$. Assume that $r\in C_b^2(\R)$ satisfies $r(0)=r'(0)=0$.
	There exists a sequence of numbers $d_\ell\to 1$ and a sequence of non-constant functions 
	$v_\ell\in W^{1,2}(\Omega)$ such that $\|v_\ell\|_{W^{1,2}} \to 0$ and which satisfy the boundary value problem
	\begin{nalign}
		\label{eq:ell:gen}
		d_\ell  \Delta_\nu v_\ell+(\mu_k+1-d_\ell) v_\ell +r(v_\ell)&=0 \qquad \text{for} \quad x\in \Omega.
	\end{nalign}
\end{lemma} 

\begin{proof}
	We prove this lemma by using the Rabinowitz Bifurcation Theorem \cite[Thm.~0.2]{MR0463990} for equations with a variational formulation.
	It is assumed in that approach that
	\begin{itemize}
		\item $E$ is a real Hilbert space,
		\item $I\in C^2(E,\R)$ with $ I'(u)=Lu+H(u)$,
		\item $L$ is linear and $H(u)=o(\|u\|)$ at $u=0$,
		\item $\mu$ is an isolated eigenvalue of $L$ of a finite multiplicity.
	\end{itemize}
	Under these assumptions, by  \cite[Thm. 0.2]{MR0463990}, the couple $(\mu,0)\in \R\times E$ is a bifurcation point of the equation 
	\begin{nalign}\label{eq:gen}
		G(\lambda ,v)\equiv Lv+H(v) -\lambda v=0
	\end{nalign}
	which means that each neighborhood of $(\mu, 0)$ contains a solution $(\lambda, v)$ with $\| v \| \neq 0$ of equation \eqref{eq:gen}. In our case, we use the the usual Sobolev space $E=W^{1,2}(\Omega)$ with the equivalent scalar product
	\begin{nalign}
		\langle u,v\rangle_\W12=\int_\Omega \nabla u\cdot \nabla v\dx+\int_\Omega uv\dx
	\end{nalign}
	as well as the functional 
	\begin{nalign}
		\label{funk}
		I(v) =\frac{\mu_k+1}{2}\int_\Omega v^2\dx +\int_\Omega R(v)\dx
	\end{nalign}
	with $R(v)=\int_0^vr(s)\,ds$. It is an elementary calculation to show that 
	\begin{itemize}
		\item $I\in C\big(W^{1,2}(\Omega),\R\big)$,
		\item it is differentiable in the Fr\'echet sense and for each $v
		\in W^{1,2}(\Omega)$
		\begin{nalign}
			DI(v) \varphi =(\mu_k+1) \int_\Omega v \varphi \dx + \int_\Omega r(v)\varphi \dx,
		\end{nalign}
		\item $DI \in C\Big(W^{1,2}(\Omega),\;$Lin$\big(W^{1,2}(\Omega),\R\big)\Big)$.
	\end{itemize}
	The second Fr\'echet derivative at the point $v\in W^{1,2}(\Omega)$ is represented by the bilinear form
	\begin{nalign}
		\big\langle D^2I(v)\varphi,\psi\big\rangle =(\mu_k+1) \int_\Omega\varphi \psi \dx + \int_\Omega r'(v)\varphi \psi \dx.
	\end{nalign}
	Let us show that $D^2 I\in C\Big(W^{1,2}(\Omega), $ Lin$ \big(W^{1,2}(\Omega), $ Lin$ (W^{1,2}(\Omega),\R)\big)\Big)$. For $v_n \to v$ in $W^{1,2}(\Omega)$ and $\psi, \varphi \in W^{1,2}(\Omega)$ we estimate
	\begin{nalign}
		\left|\big\langle\big(D^2 I(v_n) - D^2 I(v)\big)\varphi, \psi\big\rangle \right| & \leqslant  \int_\Omega |r'(v_n) - r'(v)||\varphi||\psi| \dx \\
		& \leqslant \|r''\|_\infty \int_\Omega |v_n - v| |\varphi||\psi| \dx \\ 
		& \leqslant  \|r''\|_\infty \|v_n - v\|_3 \|\varphi\|_3 \|\psi||_3  \\
		& \leqslant \|r''\|_\infty \|v_n - v\|_{W^{1,2}} \|\varphi\|_{W^{1,2}} ||\psi||_{W^{1,2}}.
	\end{nalign}
	The last inequality follows from the Sobolev embeddings with the assumption $N\leqslant 6$.

	In particular, for each test function $\varphi \in W^{1,2}(\Omega)$, we obtain
	\begin{nalign}
		I'(v)(\varphi) = (\mu_k+1)\int_\Omega v\varphi\dx +\int_\Omega r(v)\varphi \dx \equiv Lv(\varphi) +H(v)(\varphi), 
	\end{nalign}
	where, by the assumption on $r=r(v)$, we obtain immediately that 
	$H(v)=o\big(\|v\|_{W^{1,2}}\big)$ as $\|v\|_{W^{1,2}} \to 0$.

	Notice now that $\lambda=1$ is an isolated eigenvalue of finite multiplicity of the operator $L$. Indeed, this claim  is equivalent to the equality
	\begin{nalign}
		Lv(\varphi)=\langle v,\varphi\rangle_\W12 \qquad \text{for all} \quad \varphi\in W^{1,2}(\Omega),
	\end{nalign}
	that is, to the equation 
	\begin{nalign}
		(\mu_k+1)\int_\Omega v\varphi \,dx = \int_\Omega \nabla v\cdot \nabla \varphi \dx +\int_\Omega v\varphi\dx\qquad \text{for all} \quad \varphi\in W^{1,2}(\Omega),
	\end{nalign}
	which obviously reduces to the eigenvalue problem for $\Delta_\nu$. Now, we apply the fact that $\mu_k$ is an isolated eigenvalue of finite multiplicity.
	
	Thus, by the Rabinowitz  Theorem  \cite[Thm. 0.2]{MR0463990}, the couple $(1,0)$ is a bifurcation point of equation \eqref{eq:gen} which means that there exists a sequence of numbers $d_\ell\to 1$ and nonzero $\{v_\ell\}\subset  W^{1,2}(\Omega)$ such that $\|v_\ell\|_{W^{1,2}}\to 0$, satisfying the following equation 
	\begin{nalign}
		Lv_\ell(\varphi) +H(v_\ell)(\varphi) -d_\ell\langle v,\varphi\rangle_{1,2} =0\qquad \text{for all} \quad \varphi\in W^{1,2}(\Omega)
	\end{nalign}
	which, in our setting, is equivalent to the equation satisfied by the weak solutions $v_\ell\in W^{1,2}(\Omega)$ to problem \eqref{eq:ell:gen}:
	\begin{nalign}
		- d_\ell  \int_\Omega \nabla v_\ell\cdot\nabla \varphi\dx+(\mu_k+1-d_\ell) \int_\Omega v_\ell\varphi\dx +\int_\Omega r(v_\ell)\varphi\dx =0
	\end{nalign}
	for all $\varphi\in W^{1,2}(\Omega)$.
\end{proof}

\begin{proof}[Proof of Proposition \ref{thm:reg}.] 
	We apply Lemma \ref{lem:Rab} to construct non-constant solutions to the reaction-diffusion-ODE system. In the following, we denote by $B_\varepsilon(\overline{ V})$ an open ball of radius $\varepsilon>0$ centred at $\overline{ V}$. First let us prove that, in Lemma \ref{lem:Rab}, only a finite number of $v_\ell$ can be constant. Indeed, if there exists a subsequence of constants $\lbrace v_{\ell_n} \rbrace$ satisfying equation~\eqref{eq:ell:gen} such that $v_{\ell_n}\to 0$ then passing to the limit we obtain that $r'(0) =- \mu_k$ which is a contradiction. 
	
	Next, since $\det D_{\pmb u} \pmb f (\overline{\pmb U},\overline{ V})\neq0$, 
	by Implicit Function Theorem, there exists $\varepsilon>0$ and a function $\pmb k\in C^2\big(B_\varepsilon(\overline{V})\big)$ such that $\pmb k(\overline{ V})=\overline{\pmb U}$ and $\pmb f(\pmb k(v), v)=0$ for all $v \in \big(B_\varepsilon(\overline{V})\big)$. First, we show that the function $ h(v) \equiv g\big(\pmb k(v), v\big)$, defined for all  $v \in \big(B_\varepsilon(\overline{V})\big)$, satisfies the relations
	\begin{nalign}
		h(\overline{V})=0\qquad \text{and}\qquad h'(\overline{V})=\gamma\mu_k>0.
	\end{nalign}
	Indeed, the first one results from the equations $h(\overline{V})=g(\pmb k(\overline{V}),\overline{V})=g(\overline{\pmb U},\overline{V})=0$. For the second one, differentiating the function $h(v) = g\big(\pmb k(v), v\big)$   yields
	\begin{nalign}
		h^\prime (v) = D_{\pmb u} g (\pmb k(v),v)\pmb k^\prime (v)  + D_{v}  g (\pmb k(v),v). \label{th3.3-eq3}
	\end{nalign}
	On the other hand, we differentiate both sides of the equation $\pmb f \big(\pmb k(v), 	v \big) = 0$ with respect to $v$ to obtain the equality $ D_{\pmb u} \pmb f \big(\pmb k(v),v\big) \pmb k^\prime (v) + D_{v} \pmb f \big(\pmb k(v),v\big)  = 0$, or, equivalently,
	\begin{nalign}
		\pmb k^\prime (v) = - D^{-1}_{\pmb u} \pmb f (\pmb k(v),v) D_{v} \pmb f (\pmb k(v),v)  . \label{th3.3-eq4}
	\end{nalign}
	Finally, substituting equation \eqref{th3.3-eq4} into equation \eqref{th3.3-eq3}, choosing $v=\overline{V}$, and applying the notation from the statement of Theorem \ref{thm:reg} we obtain the equations
	\begin{nalign}
		h^\prime (\overline{V}) &= - D_{\pmb u} g (k(\overline{V}),\overline{ V})D^{-1}_{\pmb u} \pmb f (\pmb k(\overline{V}),\overline{ V}) D_{v} \pmb f (\pmb k(\overline{V}),\overline{ V}) + D_{v} g (k(\overline{V}),\overline{ V})\\
		& =-\pmb C_0 \pmb A^{-1}_0\pmb B_0+d_0.
	\end{nalign}
	Notice that the right-hand side satisfies the following relation
	\begin{nalign}
		\label{eq:DetIden}
		d_0 -  \bC_0 \bA ^{-1}_0 \bB_0 &=   
		\det\begin{pmatrix}
		\pmb I & \pmb 0 \\
		\pmb 0 & d_0 -  \bC_0  \bA _0^{-1} \bB_0
		\end{pmatrix} \\
		& = \frac{1}{\det \pmb A_0} \det \begin{pmatrix}
		\pmb A_0 & \pmb 0 \\
		\pmb C_0 & d_0 -  \bC_0 \bA_0 ^{-1} \bB_0
		\end{pmatrix} \\
		& = \frac{1}{\det \bA_0 }\det\left(\begin{pmatrix}
		\bA_0  & \bB_0  \\
		\bC_0  & d_0
		\end{pmatrix}\begin{pmatrix}
		\pmb{I} & -\bA_0 ^{-1}\bB_0  \\
		\textbf{0} & 1
		\end{pmatrix}\right) \\
		&= \frac{1}{\det \bA_0 }\det\begin{pmatrix}
		\bA_0  & \bB_0  \\
		\bC_0  & d_0
		\end{pmatrix}.
	\end{nalign}
	Hence, $h'(\overline{V})=\mu _k>0$.
	
	Next, we denote by $\bh$ an arbitrary extension of the function $h$ to the whole line $\R$ which satisfies    
	\begin{nalign}
		\label{h:ext}
		\bh\in C_b^2(\R) \quad\text{and}\quad
		\bh(v)=h(v)\quad  \text{for all} \quad v\in B_\varepsilon(\overline{V}).
	\end{nalign}
	By Lemma \ref{lem:Rab}, we obtain  a sequence $d_\ell\to 1$ such that the Neumann boundary value problem
	\begin{nalign}
		\label{ss:h2}
		d_\ell\Delta_\nu V_\ell+(1-d_\ell)(V_\ell-\overline{ V})+\bh(V_\ell)&=0  \quad \text{for}\quad x\in \Omega
	\end{nalign}
	has a non-constant solution $V_\ell\in W^{1,2}(\Omega)$. Indeed, it suffices to look for these solutions in the form $V_\ell = \overline{V} + z_\ell$, where  $z_\ell$ solves  the following problem
	\begin{nalign} \label{k:eq1}
		d_\ell  \Delta_\nu z_\ell + \big(\bh^\prime (\overline{V}) +1-d_\ell \big)z_\ell + r(z_\ell) &= 0 \quad \text{in}\quad \Omega, 
	\end{nalign} 
	with  $\bh^\prime (\overline{V})=\mu_k$ and  $r (z) = \bh(\overline{V}+ z)-\bh^\prime (\overline{V})z$ satisfying $r\in C_b^2(\R)$,  $r(0) = 0$, and $r^\prime (0) = 0$. 
	
	Next, by Lemma \ref{lem:Rab}, the  solutions $z_\ell$ to problem \eqref{k:eq1} satisfy $\|z_\ell\|_{W^{1,2}(\Omega)}\to 0$, hence, by the standard elliptic theory, we also have $\|z_\ell\|_{W^{2,2}(\Omega)}\to 0$. By the bootstrap arguments utilizing the elliptic $L^p$-estimates and the Sobolev embedding theorem, we conclude that $\|z_\ell\|_{W^{2,q}(\Omega)}\to 0$ for $q>N/2$ and hence $\|z_\ell\|_\Li\to 0$. In particular, if $\|z_\ell\|_\infty \leq \varepsilon $, by \eqref{h:ext}, we obtain 
	\begin{nalign}
		\bh(V_\ell)=\bh(\overline{V}+z_\ell)=h(\overline{V}+z_\ell)=h(V_\ell).
	\end{nalign}
	Thus, $V_\ell = \overline{V} + z_\ell$ is a nontrivial solution of problem \eqref{ss:h2} with $\bh$ replaced by $h$.
	
	Finally, we define $U_\ell=\pmb k(V_\ell)$ in order to obtain a nontrivial regular solution of problem \eqref{seq1:d}.
\end{proof}

\section{Linear equation and spectral analysis} 
\label{sec:LinearEquation}

\subsection{Linearised operator}
\label{sec:LinearizedSpectra}

We analyse the stability of stationary solution to problem~ \eqref{eq1}-\eqref{ini} via the usual linearisation procedure. In order to apply that approach in this section, we discuss stability properties of the zero solution to the following linear reaction diffusion-ODE system
\begin{nalign}
	\label{eq:LinearProblemDefinition}
	\pmb \varphi_t  &=   \bA  \pmb \varphi+\bB \psi, &    \quad &x\in\overline{\Omega}, \quad t>0, &&
	\\  \psi_t  &=    \Delta_\nu \psi+\bC   \pmb \varphi+d \psi,&  \quad &x\in \Omega, \quad t>0,
\end{nalign}
with matrices
\begin{nalign}
	\label{eq:LinearSystemCoefficientMatrixDefinition}
	\bA &= \pmb A(x) =
	\begin{pmatrix}
	a_{11}(x)&\dots&a_{1n}(x)\\
	\vdots&\ddots&\vdots\\
	a_{n1}(x)&\dots&a_{nn}(x)
	\end{pmatrix},
	&\bB&= \pmb B(x) =
	\begin{pmatrix}
	b_{1}(x)\\
	\vdots\\
	b_{n}(x)\\
	\end{pmatrix},
	\\
	\bC &= \pmb C(x) = 
	\begin{pmatrix}
	c_{1}(x)&\dots&c_{n}(x)
	\end{pmatrix}, 
	&d &= d(x),
\end{nalign}
where
\begin{nalign}
	\label{abcd}
	a_{ij}, \; b_i, \; c_i, \; d \in L^\infty(\Omega), \quad \text{for} \; i,j \in \lbrace 1, \cdots, n \rbrace.
\end{nalign}
In the following, we study properties of the operator $\big(\L_p, D(\L_p) \big)$ defined by the formula
\begin{nalign}
	\label{eq:LinearSystemOperatorLDefinition}
	\L_p \begin{pmatrix}
	\pmb \varphi \\ \psi
	\end{pmatrix} \equiv \begin{pmatrix}
	\pmb A \pmb \varphi+ \pmb B \psi \\ \Delta_\nu \psi + \pmb C \pmb \varphi + d \psi
	\end{pmatrix} \quad \text{with} \quad D(\L_p) = L^p(\Omega)^n \times W^{2,p}_\nu(\Omega)
\end{nalign}
for each $p\in (1,\, \infty)$, where 
\begin{nalign}
	W^{2,p}_\nu(\Omega) = \{ v\in W^{2,p}(\Omega): \, \partial_{\nu} v = 0 \text{ on } \partial \Omega \}
\end{nalign} 
is the Sobolev space $W^{2,p}(\Omega)$ supplemented with the Neumann boundary condition (see {\it e.g.}~\cite[Ch. 2, sec 2.4]{MR2573296}). 

\begin{prop}
	\label{thm:LinearOperatorLAnaliticSemigroup}
	For each $p\in (1,\infty)$ the operator $\big( \L_p, D(\L_p) \big)$ generates an analytic semigroup of linear operators on $L^p(\Omega)^{n+1}$. 
\end{prop}
\begin{proof}
	This fact is well-known in the case of the Laplace operator with the Neumann boundary condition, see \textit{e.g.} \cite[Ch. II, Sec. 5]{MR1721989}. Here, we consider  bounded perturbations of this operator.
\end{proof}

In order to study a nonlinear stability, we need a counterpart of the semigroup from Proposition \ref{thm:LinearOperatorLAnaliticSemigroup} acting on a space of bounded continuous functions. First, we recall that the operator $\big(\Delta_\nu, D(\Delta_\nu)\big)$ with the domain
\begin{nalign}
	D(\Delta_\nu) = \left\lbrace  u\in \bigcap_{p\geqslant 1} W^{2,p}_{loc}(\Omega): \quad u, \, \Delta u \in C(\overline{\Omega}), \quad  \partial_{\nu} u = 0 \text{ on } \partial\Omega  \right\rbrace
\end{nalign}
is a sectorial operator on the space $C(\overline{\Omega})$, see {\cite[Corollary 3.1.24]{MR3012216}}.

\begin{prop}
	\label{thm:LinearOperatorLAnaliticSemigroupW1p}
	If $a_{ij}, \; b_i, \; c_i, \; d \in C(\overline{\Omega})$ for $i,j \in \lbrace 1, \cdots, n \rbrace$ then the operator 
	given by the formula
	\begin{nalign}
		\label{eq:LinearSystemOperatorLInfDefinition}
		\L_\infty \begin{pmatrix}
			\pmb \varphi \\ \psi
		\end{pmatrix} \equiv \begin{pmatrix}
			\pmb A \pmb \varphi+ \pmb B \psi \\ \Delta_\nu \psi + \pmb C \pmb \varphi + d \psi
		\end{pmatrix} \quad \text{with} \quad D(\L_\infty) = C(\overline{\Omega})^n \times D(\Delta_\nu)
	\end{nalign}
	generates an analytic semigroup of linear operators on $C(\overline{\Omega})^{n+1}$.
\end{prop}

One proves this proposition in the same way as Proposition \ref{thm:LinearOperatorLAnaliticSemigroup}.

\subsection{Spectra of multiplication operators} 

As usual, we denote by $\sigma(L)$ the spectrum of a linear operator $\big(L, D(L)\big)$ and $s(L) = \sup \{\Re \lambda: \lambda \in \sigma(L)\}$ its spectral bound. In particular, for a square matrix $\pmb M$ with constant elements
\begin{nalign}
	s(\pmb M) = \max \{ \Re \lambda : \lambda \in \C \text{ is an eigenvalue of } \pmb M\}.
\end{nalign}
Now, we discuss spectra of matrix multiplication operators induced by matrices with  $x$-dependent and bounded coefficients. The following results can be found either in the work \cite{MR1484209} or in the PhD dissertation \cite{phdthesis} and we recall them together with the proofs in slightly less general versions. 

{First we show that the spectrum of the multiplication operator consist only of approximative eigenvalues.}

\begin{lemma}
	\label{thm:SpectraOfMultiplicationByMatrix}
	For $\pmb M(x) =\big( m_{i,j}(x) \big)_{i,j=1}^n$ with $m_{i,j} \in L^\infty(\Omega)$, we  define the multiplication operator on $L^p(\Omega)^n$ with $p\in [1,\infty)$ by the formula $\big( \pmb M (\cdot) \pmb \xi\big)(x) = \pmb M(x) \pmb \xi(x)$. Then its spectrum has the form 
	\begin{nalign}
		\label{eq:LinearSystemSpectrumAlternativeForm}
		\sigma\big(\pmb M(\cdot) \big) = 
		\left\{ \lambda \in \C: 
		\begin{array}{l}
			\forall{\varepsilon>0} \quad \exists{\Omega_\varepsilon \subset \Omega} \ open \ set \quad\exists{ \pmb \xi_\varepsilon \in \R^n} \\ 
			such \ that \ \| \pmb M(\cdot) \pmb \xi_\varepsilon - \lambda \pmb \xi_\varepsilon \|_{L^\infty(\Omega_\varepsilon)} \leqslant \varepsilon |\pmb \xi_\varepsilon|
		\end{array} 
		\right\}. 
	\end{nalign} 
\end{lemma}
\begin{proof}
	Denote by $\pmb M_{ess}$ the set on the right-hand side of equality \eqref{eq:LinearSystemSpectrumAlternativeForm} and let $\lambda \in \pmb M_{ess}$. Suppose that $\lambda \notin \sigma(\pmb M(\cdot))$. Hence, there a exists a constant $C>0$ such that
	\begin{nalign}
		\label{eq:SpectraOfLinearMessInvertable}
		\| \pmb \xi \|_p \leqslant C \| (\pmb M(\cdot ) - \lambda \pmb I) \pmb \xi \|_\Lp{} \quad \text{for all} \quad \pmb \xi \in L^p(\Omega)^n.
	\end{nalign}
	Now by the definition of $\pmb M_{ess}$, for each $\varepsilon>0$ there exist open set $\Omega_\varepsilon \subset \Omega$ and $\pmb \xi_\varepsilon$  such that $\| \pmb M(\cdot) \pmb \xi_\varepsilon - \lambda \pmb \xi_\varepsilon \|_{L^\infty(\Omega_\varepsilon)} \leqslant \varepsilon|\pmb \xi_\varepsilon|$. Let $\pmb \xi = \frac{1}{|\Omega_\varepsilon|}\mathds{1}_{\Omega_\varepsilon}\pmb \xi_\varepsilon$ in inequality \eqref{eq:SpectraOfLinearMessInvertable}. Then,
	\begin{nalign}
		\|\pmb \xi\|_{L^p(\Omega)} 
		&\leqslant C \|\big(\pmb M(\cdot) - \lambda\pmb I\big)\pmb \xi\|_\Lp{} \\
		&\leqslant C \|\big(\pmb M(\cdot) - \lambda\pmb I\big) \|_{L^\infty(\Omega_\varepsilon)}\|\pmb \xi \|_{L^p(\Omega_\varepsilon)}
		\leqslant C\varepsilon\|\pmb \xi \|_\Lp{}.
	\end{nalign}
	For sufficiently small $\varepsilon>0$, we have $\|\pmb \xi\|_\Lp{} = 0$ which is contradiction. 
	
	Conversely, let $\lambda\notin \pmb M_{ess}$. By the definition of $\pmb M_{ess}$, there exists $\varepsilon >0$ such that, for all $\Omega_\varepsilon \subset\Omega$ and for all vectors $\pmb \xi_\varepsilon \in \R^n$ with $| \pmb \xi_\varepsilon| = 1$ we have $\| \pmb M(\cdot) \pmb \xi_\varepsilon - \lambda \pmb \xi_\varepsilon \|_{L^\infty(\Omega_\varepsilon)} \geqslant \varepsilon$. It implies that the inequality $| \pmb M(x) \pmb \xi_\varepsilon - \lambda \pmb \xi_\varepsilon | \geqslant\varepsilon|\pmb \xi_\varepsilon|$ holds true for almost all $x\in \Omega$. Thus,
	\begin{nalign}
		\|\big(\pmb M(\cdot) - \lambda\pmb I\big)\pmb \xi\|_\Lp{} = \left(\int_\Omega \big|\big(\pmb M(x) - \lambda \pmb I\big)\pmb \xi(x) \big|^p \dx\right)^{\frac{1}{p}} 	\geqslant \varepsilon \|\pmb \xi\|_\Lp{}
	\end{nalign}
	and the operator $\pmb M(\cdot) - \lambda\pmb I$ is invertible, hence $\lambda \notin \sigma\big(\pmb M(\cdot)\big)$.
\end{proof}

The following lemma explains a relation between $\sigma\big(\pmb M(\cdot)\big)$ and the eigenvalues of the matrices $\pmb M(x)$ for each $x\in \Omega$.

\begin{lemma}
	\label{thm:SpectraOfMultiplicationByPartialyContinousFunction}
	Assume that $\pmb M(x) = \big( m_{i,j}(x)\big)_{i,j=1}^n$ with $m_{i,j} \in L^\infty (\Omega)$. Assume that there exists  a closed set $\Omega'\subset \Omega $ such that $|\Omega'|=0$ and $m_{i,j}\big|_{\Omega\setminus \Omega'}$ is continuous. Then
	\begin{nalign}
		\sigma\big(\pmb M(\cdot)\big) = \overline{\bigcup_{x\in\Omega\setminus \Omega'} \sigma\big(\pmb M(x)\big)}.
	\end{nalign} 
\end{lemma}

\begin{proof}
	 Let $\lambda \in \overline{\bigcup_{x\in\Omega\setminus \Omega'} \sigma\big(\pmb M(x)\big)}$. Hence, there exists a sequence $\{x_n\}_{n=1}^\infty \subset \Omega\setminus \Omega'$ such that the matrices $\pmb M(x_n)$ have eigenvalues satisfying $\lambda_n \to \lambda$. Denote by $\pmb \xi_{\lambda_n}$ the corresponding eigenvector, \textit{i.e.} $(\pmb M(x) - \lambda_n)\pmb \xi_{\lambda_n}  = 0$. By the assumption, matrix elements $m_{i,j}(x)$ are continuous functions on $\Omega \setminus \Omega'$. Since eigenvalues of the matrix $\pmb M(x)$ are also continues functions of $x$
	(see Remark \ref{eigen:cont}, below), for each $\varepsilon >0$ there exists a ball $B_\delta(x_n)$ such that for all $y\in B_\delta(x_n)$
	\begin{nalign}
		|(\pmb M(y) - \lambda_n)\pmb \xi_{\lambda_n} | \leqslant \varepsilon |\pmb \xi_{\lambda_n}|.
	\end{nalign}
	We take $n$ sufficiently large and apply the triangle inequality to obtain
	\begin{nalign}
		|(\pmb M(y) - \lambda)\pmb \xi_{\lambda_n} | \leqslant |(\pmb M(y) - \lambda_n)\pmb \xi_{\lambda_n} | + |(\lambda_n - \lambda)\pmb \xi_{\lambda_n} | \leqslant 	2\varepsilon |\pmb \xi_{\lambda_n}|.
	\end{nalign}
	Thus, by Lemma \ref{thm:SpectraOfMultiplicationByMatrix}, we have $\lambda \in \sigma(\pmb M(\cdot))$.
	 
	Conversely, assume that $\lambda \notin  \overline{\bigcup_{x\in\Omega\setminus \Omega'} \sigma\big(\pmb M(x)\big)}$. Thus, for each $x\in \Omega \setminus \Omega'$, the matrix $\big( \pmb M(x) - \lambda \pmb I\big)$ is invertible and there exists  $\varepsilon_x>0$ such that for all vectors $\pmb \xi_\varepsilon \in \R^n$ we have
	\begin{nalign}
	 	\label{eq:SpectraLinearConstantFunctionInvertibilityEstimate}
	 	\big| (\pmb M(x) - \lambda \pmb I)\pmb \xi_\varepsilon\big| \geqslant \varepsilon_x |\pmb \xi_\varepsilon|.
	\end{nalign}
	We show that there exists $\varepsilon >0$ such that $\varepsilon_x \geqslant \varepsilon >0$ uniformly for all $x\in \Omega \setminus \Omega'$. Indeed, suppose \textit{a contrario} that there exists a sequence $x_n \in \Omega \setminus \Omega'$ satisfying $\varepsilon_{x_n} \to 0$. We take the sequence of vectors $\{\pmb \xi_{\varepsilon_{x_n}}\}_{n=1}^\infty \subset \R^n$, with $|\pmb \xi_{\varepsilon_{x_n}}| = 1$ satisfying 
	\begin{nalign}
		\big|(\pmb M(x_n) - \lambda \pmb I)\pmb \xi_{\varepsilon_{x_n}} \big| = \varepsilon_{x_n} \to 0. 
	\end{nalign}
	Letting $n\to \infty$ we have that $\lambda$ is arbitrary close to an eigenvalue of matrix $\pmb M(x_n)$, hence $\lambda \in  \overline{\bigcup_{x\in\Omega\setminus \Omega'} \sigma\big(\pmb M(x)\big)}$. Therefore, there exists $\varepsilon >0$ such that for all $\pmb \xi_\varepsilon \in \R^n$ we have
	\begin{nalign}
		\big| (\pmb M(x) - \lambda \pmb I)\pmb \xi_\varepsilon\big| \geqslant \varepsilon |\pmb \xi_\varepsilon|, \quad \text{for all} \quad x\in \Omega\setminus \Omega'.
	\end{nalign}
	Lemma \ref{thm:SpectraOfMultiplicationByMatrix} implies again that $\lambda \notin \sigma(\pmb M(\cdot))$.
\end{proof}

\begin{rem}\label{eigen:cont}
	Notice that eigenvalues of a matrix depend continuously on the matrix   elements, see {\it eg.} \cite{MR884486}. 
Such a continuous dependence on a matrix elements may be false  for  eigenvectors,  see {\it eg.}~\cite[Example 5.3 on p.~111]{MR1335452}.
\end{rem}

\subsection{Spectrum of the reaction-diffusion-ODE operator}%

\begin{theorem}
	\label{thm:SpectrumOfOperatorL}
	Let $\big(\L_p, D(\L_p)\big)$, with $p\in(1,\infty)$, be the operator given by formula~\eqref{eq:LinearSystemOperatorLDefinition} with matrices~\eqref{eq:LinearSystemCoefficientMatrixDefinition}-\eqref{abcd}. 
	Denote by $\sigma\big( \pmb A(\cdot) \big)$ be the spectrum of multiplication operator $\pmb A(\cdot)$ on $L^p(\Omega)^n$. Then, the spectrum of $\big(\L_p, D(\L_p)\big)$ can be characterized as
	\begin{nalign}
		\sigma(\L_p) = \sigma\big(\pmb A(\cdot)\big) \cup \lbrace \text{eigenvalues of } \; \L_p \rbrace.
	\end{nalign}
	In particular, the spectrum $\sigma(\L_p)$ is independent of $p$. 
\end{theorem}

\begin{proof}
	We improve the reasoning form \cite[Thm 4.5]{MR3600397}, where one ODE coupled with one PDE were considered.
	First, we prove that $\sigma\big(\pmb A(\cdot) \big) \subset \sigma(\L_p)$ by showing that for each $\lambda \in \sigma\big(\pmb A(\cdot) \big)$ the operator 
	\begin{nalign}
		\L_p -\lambda \pmb I: L^p(\Omega)^n\times \W2p \rightarrow L^p(\Omega)^n\times L^p(\Omega)
	\end{nalign}
	defined by formula
	\begin{nalign}
		(\L_p - \lambda \pmb I)\begin{pmatrix}\pmb \varphi\\\psi\end{pmatrix} = \begin{pmatrix} (\pmb A - \lambda)\pmb \varphi + \pmb B \psi \\ \Delta_\nu \psi + \pmb C \pmb \varphi + (d - \lambda) \psi
		\end{pmatrix}
	\end{nalign}
	cannot have a bounded inverse. Indeed, suppose {\it a contrario} that $(\L_p - \lambda \pmb I)^{-1}$ exists { and is a bounded operator from $(L^p(\Omega))^{n+1}$ into $D(\L_p)$ equipped with the graph norm. The graph norm of $\L_p$ is equivalent to the norm on the product space $(L^p(\Omega))^n \times (\W2p)$ by \cite[Section 2.4]{MR2573296}}.~Thus, there exists a number  $K >0$ such that
	\begin{nalign}
		\label{eq:InstabilityPsiPhiEstimate}
		\|\pmb \varphi \|_\Lp{} + \|\psi\|_{W^{2,p}(\Omega)} \leqslant K \big( &\|  (\pmb A  - \lambda \pmb I) \pmb \varphi + \pmb B \psi \|_\Lp{} \\  +& \| \Delta_\nu \psi +  \pmb C   \pmb \varphi + (d - \lambda) \psi \|_\Lp{} \big)
	\end{nalign}
	 for all $(\pmb \varphi, \psi) \in L^p(\Omega)^{n}\times \W2p$. Let $\varepsilon>0$ be arbitrary (to be chosen later on). By Lemma  \ref{thm:SpectraOfMultiplicationByMatrix}  there exists $\pmb \xi \in \R^n$ with $|\pmb \xi| = 1$ and a ball $B_\varepsilon$ such that
	 \begin{nalign}
	 	\|(\pmb A(\cdot) - \lambda \pmb I)\pmb\xi \|_{L^\infty(B_\varepsilon)}\leqslant\varepsilon.
	 \end{nalign}
	 Let us show that, for arbitrary $\psi \in C_c^\infty(\Omega)$  with  $\text{supp} \;\psi \subset B_\varepsilon $, we may find $\pmb \varphi \in L^p(\Omega)^n$ with the following properties
	 \begin{itemize}
	 	\item $\pmb \varphi =  \rho \pmb \xi$ for some $\rho\in \Lp{}$ with   $\text{supp}\, \rho \subset B_\varepsilon $, 
	 	\item $\zeta = \Delta_\nu \psi + \pmb C \pmb \varphi + (d- \lambda)\psi $ satisfies $\| \zeta\|_\Lp{} \leqslant \varepsilon \| \pmb\varphi \|_\Lp{}$.
	 \end{itemize}
	Indeed, we cut the function $\pmb C$ at the level $\varepsilon$ in the following way
	\begin{nalign}
		\pmb C^\varepsilon = \pmb C^\varepsilon (x) \equiv 
		\begin{cases}
			 \pmb C(x) & \text{if} \quad  \big|\pmb C(x)\pmb \xi \big| > \varepsilon, \\
			(\varepsilon,\cdots, \varepsilon)  & \text{if} \quad  \big|\pmb C(x)\pmb \xi \big| \leqslant \varepsilon.
		\end{cases}
	\end{nalign}
	Thus we obtain 
	\begin{nalign}
		\zeta = \Delta_\nu \psi + \pmb C \pmb \varphi + (d - \lambda) \psi = \Delta_\nu \psi + \pmb C^\varepsilon \pmb \varphi + (d- \lambda) \psi + (\pmb C - \pmb C^\varepsilon) 	\pmb \varphi 
	\end{nalign}
	for
	\begin{nalign}
		\pmb \varphi =\frac{-(\Delta_\nu \psi +(d - \lambda) \psi)}{\pmb C^\varepsilon \pmb \xi}\pmb \xi = \rho \pmb \xi \in L^p(\Omega)^n \quad \text{and} \quad \zeta = (\pmb 	C - \pmb C^\varepsilon)\pmb \varphi \in L^p(\Omega)
	\end{nalign}
	with $\| (\pmb C - \pmb C^\varepsilon ) \pmb \xi \|_{L^\infty(\Omega)} \leqslant n\varepsilon$.
	
	Now, noting that supp $\pmb \varphi \subset B_\varepsilon$ and $\pmb \varphi(x) = \rho(x) \pmb \xi$, we obtain the inequality
	\begin{nalign}
		\| (\pmb A - \lambda \pmb I) \pmb \varphi  \|_\Lp{} \leqslant \| (\pmb A - \lambda \pmb I) \pmb \xi \|_{L^\infty (B_\varepsilon)} \|\rho\|_\Lp{} \leqslant \varepsilon \|\pmb 	\varphi\|_{\Lp{}}.
	\end{nalign}
	Thus, substituting the functions $\pmb \varphi, \psi, \pmb\xi$ into inequality \eqref{eq:InstabilityPsiPhiEstimate}, we obtain the estimate 
	\begin{nalign}
		\label{eq:LinearSystemOperatorSpectraPsiEstimate}
		\| \pmb \varphi \|_{\Lp{}} &+ \| \psi \|_{W^{2,p}(\Omega)} \\ 
		&\leqslant K \big( \| ( \pmb A - \lambda \pmb I) \pmb \varphi \|_{\Lp{}} + \| \pmb B \psi \|_{\Lp{}} +  \| \zeta \|_\Lp{} \big) \\
		&\leqslant K\big( (n+1) \varepsilon   \|\pmb \varphi \|_{\Lp{}} + \| \pmb B \|_\Li \| \psi \|_\Lp{} \big).
	\end{nalign}
	Hence, choosing $\varepsilon >0$ small enough we compensate the term $K\varepsilon(n+1) \| \pmb \varphi \|_\Lp{}$ by its counterpart on the left-hand side of inequality \eqref{eq:LinearSystemOperatorSpectraPsiEstimate} and obtain estimates
	\begin{nalign}
		\| \psi \|_{W^{2,p}(\Omega)} \leqslant \big(1 - K(n+1) \varepsilon\big) \|\pmb \varphi \|_\Lp{} + \| \psi \|_\W2p \leqslant K \| \pmb B \|_\Li \| \psi \|_\Lp{}
	\end{nalign}
	which obviously can not hold for all $\psi \in C_c^\infty(\Omega)$ such that supp $\psi \subset B_\varepsilon$. This completes the proof that $\sigma\big(\pmb A(\cdot)\big) \subset \sigma(\L_p)$. Moreover, by Lemma \ref{thm:SpectraOfMultiplicationByPartialyContinousFunction}, the spectrum $\sigma\big(\pmb A (\cdot) \big)$ is independent of $p$. 
	
	Next, we prove that the remainder of the spectrum $\big(\mathcal{L}_p, D(\mathcal{L}_p)\big)$ consists of a discrete set of eigenvalues by analysing the corresponding resolvent equations
	\begin{nalign}
		\label{eq:LinearSystemDiscreteSpectrumOfL}
		(\pmb A - \lambda \pmb I) \pmb \varphi + \pmb B \psi &= \pmb F, \quad  x\in \overline{\Omega}, \\
		 \Delta_\nu \psi + \pmb C \pmb \varphi + (d-\lambda)\psi &= G, \quad  x\in \Omega
	\end{nalign}
	with arbitrary $(\pmb F, G)\in L^p(\Omega)^n\times L^p(\Omega)$. Since, $\lambda \notin \sigma(\pmb A(\cdot))$ we can solve the first equation with respect to $\pmb \varphi$ and substitute into the second equation to obtain the boundary value problem
	\begin{nalign}
		\label{eq:LinearSystemSpectrumLInvertibility}
		\Delta_\nu \psi + q(\lambda) \psi = r(\lambda)
	\end{nalign}
	where
	\begin{nalign}
		\label{eq:LinearSystemSpectrumLPQFunctions}
		q(\lambda) &= q(x,\lambda) = -\pmb C(x) \big( \pmb A(x) -\lambda\pmb I \big)^{-1}\pmb B(x) + d(x) - \lambda, \\
		r(\lambda) &= r(x,\lambda) = -\pmb C(x) \big( \pmb A(x) -\lambda\pmb I \big)^{-1} \pmb F(x) + G(x). 
	\end{nalign}
	By the Analytic Fredholm Theorem for Banach spaces from \cite{MR233240}, used in the same way as in \cite[proof of Thm~4.5]{MR3600397} we conclude that
	the boundary value problem \eqref{eq:LinearSystemSpectrumLInvertibility}
	has no unique solution only for a countable set of $\lambda\in\C$.
	Eigenfunctions corresponding to those eigenvalues satisfy elliptic equation \eqref{eq:LinearSystemSpectrumLInvertibility}, hence, by standard elliptic regularity, they belong to $\W2p$ for each $p\in(1,\infty)$. This means that the set of eigenvalues of $\L_p$ is independent of $p$.
\end{proof}

\begin{rem}
	\label{thm:SigmaW1p}
	By Theorem \ref{thm:SpectrumOfOperatorL}, the spectrum of the operator $\big(\L_\infty, D(\L_\infty)\big)$ given by formula \eqref{eq:LinearSystemOperatorLInfDefinition} with the matrix coefficients 
	\begin{nalign}
		\label{eq:abcdW1p}
		a_{ij}, \; b_i, \; c_i, \; d \in C(\overline\Omega), \quad \text{for} \; i,j \in \lbrace 1, \cdots, n \rbrace
	\end{nalign}
	satisfies	
	\begin{nalign}
		\sigma\left( \L_p\right)=\sigma\big(\pmb A(\cdot)\big) \cup \lbrace \text{eigenvalues of } \; \L_{p} \rbrace \subset \sigma(\L_\infty).
	\end{nalign}
	Indeed, by the elliptic regularity, eigenfunctions of $-\Delta_\nu$ belong to $\W2p$ for each $p\in(1,\infty)$, thus they are continuous up to the boundary.
	Using analogous  estimates  as those in the proof of Theorem \ref{thm:SpectrumOfOperatorL} with $L^p$-norm replaced by the $L^\infty$-norm on $C(\overline\Omega)$, we obtain that  
$	
	\sigma\big(\pmb A(\cdot)\big)  \subset \sigma(\L_\infty).
$
\end{rem}

\section{Instability of regular steady states} 
\label{sec:InstRes}

We are in a position to prove the main results of this work. 

\begin{proof}[Proof of Theorem \ref{thm:Aposit}]
	Let $\UV$ be a regular stationary solution (either constant or non-constant). We consider the operator $\big(\L_p, D(\L_p) \big)$ given by formula \eqref{eq:LinearSystemOperatorLDefinition} and the matrices 
	\begin{nalign}
		\label{eq:MatrixDefInstability}
		\pmb A(x) &= D_{\pmb u} \pmb f ({\pmb U(x)},{ V(x)}), &  \pmb B(x) &= D_{v} \pmb f ({\pmb U(x)},{ V(x)}), \\ 
		\pmb C(x) &= D_{\pmb u} g ({\pmb U(x)},{ V(x)}), 	  &	      d(x) &= D_{v}  g ({\pmb U(x)},{ V(x)}).
	\end{nalign} 
	Under the assumption $s\big(D_{\pmb u} \pmb f ({\pmb U(x)},{ V(x)})\big) >0$ we obtain that $s(\L_p)>0$ by Theorem~\ref{thm:SpectrumOfOperatorL}. Since $\UV$ is a regular solution, the matrix coefficients belong to the space $C(\overline\Omega)$ (see Remark \ref{thm:RegSolProp}). Consequently, Remark \ref{thm:SigmaW1p} provides $s(\L_{\infty})>0$   which implies that $\UV$ is linearly unstable in $C(\overline\Omega)$. The nonlinear instability in the Lyapunov sense of this solution is an immediate consequence of Theorem \ref{thm:NonLinInst} from Appendix.
\end{proof}

\begin{proof}[Proof of Theorem \ref{thm:non-const}] 
	Now, we consider non-constant regular solution $\UVx$ such that $s\big( f_{\pmb u} (\pmb U(\cdot), \, V(\cdot) )\big)\leqslant 0$ and $		\det \pmb f_{\pmb u} \big(\pmb U(x), \, V(x) \big) \neq 0$ for all $x\in \overline{\Omega}$. Moreover, we assume that $\Omega$ is convex.	Here, we are inspired by methods from the papers \cite{MR480282, MR555661} where instability of non-constant solutions to one reaction-diffusion equation was shown. We consider the operator $\big(\L_2, D(\L_2)\big)$ given by formula \eqref{eq:LinearSystemOperatorLDefinition} with the matrices \eqref{eq:MatrixDefInstability} and we study the eigenvalue problem
	\begin{nalign}
		\label{t1}
		\pmb A \pmb \varphi + \pmb B \psi &= \lambda \pmb \varphi, \\
		\Delta_\nu \psi + \pmb C \pmb \varphi + d\psi &= \lambda \psi.
	\end{nalign}
	
It follows from assumptions 
\eqref{eq:sLeZero} and \eqref{eq:ConZeroDet}
combined with Lemma \ref{thm:SpectraOfMultiplicationByPartialyContinousFunction} that 
 $\det \big(\pmb A(x) - \lambda \pmb I\big) \neq 0$
 for all $\lambda \geqslant 0$ and all $x\in \overline{\Omega}$,
 hence, 
 the matrix $(\pmb A(x) - \lambda \pmb I)$ is invertible. Thus, we substitute $\pmb \varphi$ computed from the first equation in system \eqref{t1} into the second equation which takes the following form 
	\begin{nalign}
		\Delta_\nu \psi + \Big(-\pmb C  (\pmb A - \lambda \pmb I)^{-1} \pmb B + d\Big)\psi = \lambda \psi .
	\end{nalign}
	For all $\lambda \geqslant 0 $ we define the family of the operators 
	\begin{nalign}
		G(\lambda) &: \psi \mapsto 	\Delta_\nu \psi + \Big(-\pmb C  (\pmb A - \lambda \pmb I)^{-1} \pmb B + d\Big)\psi
	\end{nalign}
	with the domain $D\big(G(\lambda) \big)= \W22$. We study their eigenvalues, namely, the numbers $\eta  \in \C$ for which the equation 
	\begin{nalign}
		G(\lambda) \psi = \eta \psi
	\end{nalign}
	has a nonzero solution and we are going to show that $\eta = \lambda$ for some $\lambda >0$. This will give a positive eigenvalue of system~\eqref{t1} because, under our assumption, system~\eqref{t1} is equivalent to the equation $G(\lambda) \psi = \lambda \psi$. 
	
	First we show that
	\begin{nalign}
		\label{t2}
		\sup_{\substack{x\in \overline{\Omega} \\ \lambda \geqslant 0}} \left| -\pmb C(x)  (\pmb A(x) - \lambda \pmb I)^{-1} \pmb B(x) + d(x) \right|  < \infty.
	\end{nalign}
	Let $\lambda \geqslant 0$. By assumption \eqref{eq:ConZeroDet}, the continuity of the determinant, and regularity of solutions, there exists $\overline{\varepsilon} >0$ such that every $\rho \in \sigma \big(\pmb A (\cdot) -\lambda \pmb I \big)$ satisfies $|\rho| \geqslant \overline{\varepsilon}$. Consequently, $|\det (\pmb A(x)  - \lambda \pmb I)| \geqslant \left(\max \lbrace \overline{\varepsilon}, \lambda \rbrace\right)^n$ for every $x\in \overline{\Omega}$ and hence
	\begin{nalign}
		\|(\pmb A(x) - \lambda \pmb I)^{-1}\| &=  \frac{1}{|\det (\pmb A(x) - \lambda \pmb I)|}\left\| {\text{adj}(\pmb A(x) - \lambda\pmb I)} \right\| \\ &\leqslant \frac{{\|\text{adj}(\pmb A(\cdot) - \lambda\pmb I)\|_\infty}}{\left(\max\lbrace \overline{\varepsilon}, \lambda \rbrace\right)^n} \leqslant \frac{C(\lambda^n+1)}{\left(\max\lbrace \overline{\varepsilon}, \lambda \rbrace\right)^n} \leqslant C.
	\end{nalign}
		
	Now, denote by $\eta_0(\lambda)$ the largest eigenvalue of problem $G(\lambda) \psi = \eta_0(\lambda) \psi.$
	The operator $\big(G, D(G)\big)$ is self-adjoint, hence, the eigenvalue $\eta_0(\lambda)$ is given by Rayleigh quotient
	\begin{nalign}
		\label{Gsup}
		\eta_0(\lambda) &= \sup_{\substack{ \psi \in \W12\\ \|\psi\|_\Lp2 = 1}} \big\langle G(\lambda) \psi, \psi  \big \rangle \\ &= \sup_{\substack{ \psi \in \W12\\ \|\psi\|_\Lp2 = 1}} \left(- \int_\Omega |\nabla \psi|^2\dx + \int_\Omega \Big(-\pmb C  (\pmb A - \lambda \pmb I)^{-1} \pmb B + d\Big) \psi^2  \dx \right).
	\end{nalign} 
	By inequality \eqref{t2}, we immediately obtain that $\sup_{\lambda\geqslant 0 }\eta_0(\lambda) <\infty$. In the next step of this proof, we show  that $\eta_0(0)>0$. 
	
	In  the one dimensional case $\Omega=(a,b)\subset\R$, differentiating equations \eqref{seq1} we obtain
	\begin{nalign}
		\pmb A \pmb U_x+\pmb B V_x=0 \quad \text{and}\quad (V_x)_{xx} + \pmb C \pmb U_x+dV_x=0
	\end{nalign}
	which yields
	\begin{nalign}
		\label{e-lin-1}
		(V_x)_{xx} + (-\pmb C  \pmb A^{-1} \pmb B + d) V_x  = 0.
	\end{nalign}
	Choosing $\psi = V_x/\|V_x\|_\Lp{2}$ in the functional in the right hand side of equation \eqref{Gsup}, integrating by parts, and using relations \eqref{e-lin-1} we obtain the inequality 
	\begin{nalign}\label{Lpp}
		\eta_0(0) \geqslant \frac{1}{\|V_x\|_\Lp{2}} \int_\Omega - V_{xx}V_{xx} 
		+ (-\pmb C \pmb A ^{-1} \pmb B + d)V_x^2 \dx=0.
	\end{nalign}
	The hypothesis $\eta_0(0)=0$ implies that $V_x$ is a maximizer of the functional in the right hand side of equation \eqref{Gsup} with $\lambda = 0$, hence, $V_x$ satisfies equation \eqref{e-lin-1} with Neumann boundary conditions, namely, $V_{xx}(x) = 0$ for $x\in \{a,b\}=\partial \Omega$. This ordinary differential equation for $V_x$ supplemented with the boundary conditions 
	\begin{nalign} 
		V_x(x)=0 \quad \text{and} \quad V_{xx}(x) = 0 \quad  \text{for} 
		\quad x\in \{a,b\}=\partial \Omega
	\end{nalign}
	has the unique solution $V_x\equiv 0$. Hence, $V$ (and so $U$) is a constant function which is a contradiction. Thus, $\eta_0(0) > 0 $. 
	
	The proof in higher dimensions is analogous. Differentiating equations \eqref{seq1} with respect to $x_j$ for $j\in \{1,...,n\}$, we obtain
	\begin{nalign}
		\label{e-lin-1n}
		\pmb A \pmb U_{x_j}+\pmb B V_{x_j}=0 \quad \text{and}\quad \Delta V_{x_j} + \pmb C \pmb U_{x_j}+dV_{x_j}=0.
	\end{nalign}
	Thus, by direct calculations, analogous to those in  \eqref{Lpp} and by the divergence theorem, we obtain
	\begin{nalign}\label{sum:L:1}
		\sum_{j=1}^N \langle G(0) V_{x_j}, V_{x_j} \rangle = -\frac12 \int_{\partial \Omega} \partial_\nu |\nabla V|^2 \,d\sigma.
	\end{nalign}
	Now, we use the use the inequality $\partial_\nu |\nabla V|^2\leqslant 0$ on $\partial \Omega $ which is valid for every $V\in C^2(\overline \Omega)$ and for every  convex $\Omega\subset \R^n$. Detailed calculations leading to this inequality can be found either in \cite[p. 269]{MR480282} or in \cite[Lemma 5.3]{MR555661}. Thus, at least one term on the left hand side of equality \eqref{sum:L:1}, is nonnegative, which implies 
	that   $\eta_0(0)\geqslant 0$ in the variational formula \eqref{Gsup}. If $\eta_0(0)=0$, by a variational argument analogously as in the one dimensional case, equations 
	\eqref{e-lin-1n} have to be satisfied by the function $V_{x_{j_0}}$ together with the Neumann boundary conditions for $V_{x_{j_0}}$. Applying the Hopf maximum principle (see e.g. \cite[Theorems  5 and 7 in Sec.~3]{MR762825}), we obtain immediately that $V_{x_{j_0}}\equiv 0$ which is a contradiction because we deal with the non-constant solution $V$. Hence, we have got $\eta_0(0)>0$. 
	
	Therefore, since $\eta_0(\lambda)$ is a continuous and bounded function on $[0,\infty)$  satisfying $\eta_0(0) > 0$, there exists $\overline{\lambda}>0$ such that  $\eta_0(\overline \lambda) = \overline{\lambda}$.  The number $\overline{\lambda}$ is a positive eigenvalue of system \eqref{t1} and hence $s(\L_2)>0$. Since $\UV$ is a regular solution, we conclude by Remark \ref{thm:SigmaW1p} that the number $\overline{\lambda}$ is also an eigenvalue of $\L_\infty$. Since $s(\L_\infty) > 0$, the nonlinear instability follows immediately from Theorem \ref{thm:NonLinInst} in Appendix.
\end{proof}

\appendix 

\section{Linearisation principle}

In order to show that a linear instability of regular stationary solutions of problem~\eqref{eq1}-\eqref{ini} implies their nonlinear instability in the Lyapunov sense, we recall the reasoning from \cite[Thm 2.1]{MR3600397}. That approach concerns general evolution equation
\begin{equation}\label{eq:a}
	w_t=\L w +\N(w), \qquad w(0)=w_0,
\end{equation}
where $\L$ is the generator of a $C_0$-semigroup of 
linear operators  $\{e^{t\L}\}_{t\geq 0}$  on a Banach space $X$
and $\N$ is  a nonlinear operator such that $ {\N(w)} = o\left(\|{w}\|\right)$ as $w\to 0$. Here, we recall one of the possible instability results, where the existence of a positive part of  the spectrum of the linear operator $\L$ is sufficient to show  that the zero solution of equation \eqref{eq:a}
is unstable. 

\begin{theorem}[{\cite[Thm 1]{MR1752509}}]\label{thm:SS}
	Consider an abstract problem \eqref{eq:a}, where
	\begin{enumerate}
		\item the linear operator $\L$ generates a strongly continuous semigroup of linear operators on a Banach space $X$,
		\item the intersection of the spectrum of $\L$ with the right half-plane $\{\lambda\in \C:\; \Re \lambda>0\}$ is nonempty.   
		\item $\N :X\to X$ is continuous and there exist constants $\rho>0$, $\eta>0$, and $C>0$ 
		such that $\|\N (w)\|_X \leq C\|w\|_X^{1+\eta}$ for all $\|w\|_X<\rho$. \label{Con3}
	\end{enumerate} 
	Then the zero solution of this equation is (nonlinearly) unstable.
\end{theorem}

\begin{theorem}
	\label{thm:NonLinInst}
	Let $\UV$ be a regular stationary solution of problem \eqref{eq1}-\eqref{ini}. Assume that the corresponding linearisation operator $\L_{\infty}$, defined in Proposition \ref{thm:LinearOperatorLAnaliticSemigroupW1p} and with matrices \eqref{eq:MatrixDefInstability}, satisfies $s(\L_{\infty}) > 0$. Then $\UV$ is nonlinearly unstable in the space $C(\overline\Omega)^{n+1}$.
\end{theorem}

\begin{proof}
	Applying the usual linearisation procedure we obtain that the perturbation of a stationary solution $\UV$ satisfies the semi-linear problem
	\begin{nalign}
		\label{eq:NonlinearStabilityLinearisedProblemWithRest}
		\frac{\partial}{\partial t}
		\begin{pmatrix}
			\pmb \varphi \\ \psi 
		\end{pmatrix} = \L_{\infty} 
		\begin{pmatrix}
			\pmb \varphi \\ \psi
		\end{pmatrix} + 
		\begin{pmatrix}
			\pmb R_1(\pmb \varphi, \psi) \\ R_2(\pmb \varphi, \psi)
		\end{pmatrix}
	\end{nalign}
	with the operator $\L_{\infty}$ given by formula  \eqref{eq:LinearSystemOperatorLDefinition} with matrices \eqref{eq:MatrixDefInstability} and where 
	the remainders $\pmb R_1$ and $R_2$ are obtained from nonlinearities $\pmb f, \, g$ via the Taylor expansion. We apply Theorem \ref{thm:SS} with the Banach space $	X=C(\overline\Omega)^{n+1}$. The remainders $\pmb R_1,\, R_2$ in system~\eqref{eq:NonlinearStabilityLinearisedProblemWithRest} satisfy the condition~\eqref{Con3} of Theorem \ref{thm:SS} with $\eta=1$.
\end{proof}




\begin{thebibliography}{10}
	
	\bibitem{MR480282}
	{\sc R.~G. Casten and C.~J. Holland}, {\em Instability results for reaction
		diffusion equations with {N}eumann boundary conditions}, J. Differential
	Equations, 27 (1978), pp.~266--273.
	
	\bibitem{MR2297947}
	{\sc L.~H. Chuan, T.~Tsujikawa, and A.~Yagi}, {\em Asymptotic behavior of
		solutions for forest kinematic model}, Funkcial. Ekvac., 49 (2006),
	pp.~427--449.
	
	\bibitem{CMCKS02}
	{\sc S.~Cygan, A.~Marciniak-Czochra, G.~Karch, and K.~Suzuki}, {\em Stable
		discontinuous stationary solutions to reaction-diffusion-{ODE} systems},
	preprint in arXiv,  (2021).
	
	\bibitem{MR1484209}
	{\sc K.-J. Engel}, {\em Operator matrices and systems of evolution equations},
	Surikaisekikenkyusho Kokyuroku,  (1996), pp.~61--80.
	\newblock Nonlinear evolution equations and their applications (Japanese)
	(Kyoto, 1995).
	
	\bibitem{MR1721989}
	{\sc K.-J. Engel and R.~Nagel}, {\em One-parameter semigroups for linear
		evolution equations}, vol.~194 of Graduate Texts in Mathematics,
	Springer-Verlag, New York, 2000.
	\newblock With contributions by S. Brendle, M. Campiti, T. Hahn, G. Metafune,
	G. Nickel, D. Pallara, C. Perazzoli, A. Rhandi, S. Romanelli and R.
	Schnaubelt.
	
	\bibitem{MR2833346}
	{\sc Y.~Golovaty, A.~Marciniak-Czochra, and M.~Ptashnyk}, {\em Stability of
		nonconstant stationary solutions in a reaction-diffusion equation coupled to
		the system of ordinary differential equations}, Commun. Pure Appl. Anal., 11
	(2012), pp.~229--241.
	
	\bibitem{MR884486}
	{\sc G.~Harris and C.~Martin}, {\em The roots of a polynomial vary continuously
		as a function of the coefficients}, Proc. Amer. Math. Soc., 100 (1987),
	pp.~390--392.
	
	\bibitem{MR3214197}
	{\sc S.~H\"{a}rting and A.~Marciniak-Czochra}, {\em Spike patterns in a
		reaction-diffusion {ODE} model with {T}uring instability}, Math. Methods
	Appl. Sci., 37 (2014), pp.~1377--1391.
	
	\bibitem{MR3583499}
	{\sc S.~H\"{a}rting, A.~Marciniak-Czochra, and I.~Takagi}, {\em Stable patterns
		with jump discontinuity in systems with {T}uring instability and hysteresis},
	Discrete Contin. Dyn. Syst., 37 (2017), pp.~757--800.
	
	\bibitem{MR684081}
	{\sc Y.~Hosono and M.~Mimura}, {\em Singular perturbations for pairs of
		two-point boundary value problems of {N}eumann type}, in Mathematical
	analysis on structures in nonlinear phenomena ({T}okyo, 1978), vol.~2 of
	Lecture Notes Numer. Appl. Anal., Kinokuniya Book Store, Tokyo, 1980,
	pp.~79--138.
	
	\bibitem{MR1335452}
	{\sc T.~Kato}, {\em Perturbation theory for linear operators}, Classics in
	Mathematics, Springer-Verlag, Berlin, 1995.
	\newblock Reprint of the 1980 edition.
	
	\bibitem{MR791838}
	{\sc K.~Kishimoto and H.~F. Weinberger}, {\em The spatial homogeneity of stable
		equilibria of some reaction-diffusion systems on convex domains}, J.
	Differential Equations, 58 (1985), pp.~15--21.
	
	\bibitem{Kthe2020HysteresisdrivenPF}
	{\sc A.~K\"{o}the, A.~Marciniak-Czochra, and I.~Takagi}, {\em Hysteresis-driven
		pattern formation in reaction-diffusion-{ODE} systems}, Discrete Contin. Dyn.
	Syst., 40 (2020), pp.~3595--3627.
	
	\bibitem{phdthesis}
	{\sc C.~Kowall}, {\em Uniform Shadow Limit Reduction for Reaction-Diffusion-ODE
		Systems}, PhD thesis, University of Heidelberg, 01 2021.
	
	\bibitem{MR3679890}
	{\sc Y.~Li, A.~Marciniak-Czochra, I.~Takagi, and B.~Wu}, {\em Bifurcation
		analysis of a diffusion-{ODE} model with {T}uring instability and
		hysteresis}, Hiroshima Math. J., 47 (2017), pp.~217--247.
	
	\bibitem{MR3973251}
	\leavevmode\vrule height 2pt depth -1.6pt width 23pt, {\em Steady states of
		{F}itz{H}ugh-{N}agumo system with non-diffusive activator and diffusive
		inhibitor}, Tohoku Math. J. (2), 71 (2019), pp.~243--279.
	
	\bibitem{MR3012216}
	{\sc A.~Lunardi}, {\em Analytic semigroups and optimal regularity in parabolic
		problems}, Modern Birkh\"{a}user Classics, Birkh\"{a}user/Springer Basel AG,
	Basel, 1995.
	\newblock [2013 reprint of the 1995 original] [MR1329547].
	
	\bibitem{MR2205561}
	{\sc A.~Marciniak-Czochra}, {\em Receptor-based models with hysteresis for
		pattern formation in hydra}, Math. Biosci., 199 (2006), pp.~97--119.
	
	\bibitem{MR3059757}
	\leavevmode\vrule height 2pt depth -1.6pt width 23pt, {\em Reaction-diffusion
		models of pattern formation in developmental biology}, in Mathematics and
	life sciences, vol.~1 of De Gruyter Ser. Math. Life Sci., De Gruyter, Berlin,
	2013, pp.~191--212.
	
	\bibitem{MR3039206}
	{\sc A.~Marciniak-Czochra, G.~Karch, and K.~Suzuki}, {\em Unstable patterns in
		reaction-diffusion model of early carcinogenesis}, J. Math. Pures Appl. (9),
	99 (2013), pp.~509--543.
	
	\bibitem{MR3600397}
	\leavevmode\vrule height 2pt depth -1.6pt width 23pt, {\em Instability of
		{T}uring patterns in reaction-diffusion-{ODE} systems}, J. Math. Biol., 74
	(2017), pp.~583--618.
	
	\bibitem{MR3345329}
	{\sc A.~Marciniak-Czochra, M.~Nakayama, and I.~Takagi}, {\em Pattern formation
		in a diffusion-{ODE} model with hysteresis}, Differential Integral Equations,
	28 (2015), pp.~655--694.
	
	\bibitem{MR555661}
	{\sc H.~Matano}, {\em Asymptotic behavior and stability of solutions of
		semilinear diffusion equations}, Publ. Res. Inst. Math. Sci., 15 (1979),
	pp.~401--454.
	
	\bibitem{MR579554}
	{\sc M.~Mimura, M.~Tabata, and Y.~Hosono}, {\em Multiple solutions of two-point
		boundary value problems of {N}eumann type with a small parameter}, SIAM J.
	Math. Anal., 11 (1980), pp.~613--631.
	
	\bibitem{MR2527521}
	{\sc G.~Mulone and V.~A. Solonnikov}, {\em Linearization principle for a system
		of equations of mixed type}, Nonlinear Anal., 71 (2009), pp.~1019--1031.
	
	\bibitem{MR2103689}
	{\sc W.-M. Ni}, {\em Qualitative properties of solutions to elliptic problems},
	in Stationary partial differential equations. {V}ol. {I}, Handb. Differ.
	Equ., North-Holland, Amsterdam, 2004, pp.~157--233.
	
	\bibitem{perthame2020fast}
	{\sc B.~Perthame and J.~Skrzeczkowski}, {\em Fast reaction limit with
		nonmonotone reaction function}, arXiv:2008.11086,  (2020).
	
	\bibitem{MR762825}
	{\sc M.~H. Protter and H.~F. Weinberger}, {\em Maximum principles in
		differential equations}, Springer-Verlag, New York, 1984.
	\newblock Corrected reprint of the 1967 original.
	
	\bibitem{MR0463990}
	{\sc P.~H. Rabinowitz}, {\em A bifurcation theorem for potential operators}, J.
	Functional Analysis, 25 (1977), pp.~412--424.
	
	\bibitem{MR845785}
	\leavevmode\vrule height 2pt depth -1.6pt width 23pt, {\em Minimax methods in
		critical point theory with applications to differential equations}, vol.~65
	of CBMS Regional Conference Series in Mathematics, Published for the
	Conference Board of the Mathematical Sciences, Washington, DC; by the
	American Mathematical Society, Providence, RI, 1986.
	
	\bibitem{MR1036472}
	{\sc K.~Sakamoto}, {\em Construction and stability analysis of transition layer
		solutions in reaction-diffusion systems}, Tohoku Math. J. (2), 42 (1990),
	pp.~17--44.
	
	\bibitem{MR727393}
	{\sc R.~Schaaf}, {\em Global behaviour of solution branches for some {N}eumann
		problems depending on one or several parameters}, J. Reine Angew. Math., 346
	(1984), pp.~1--31.
	
	\bibitem{MR808736}
	\leavevmode\vrule height 2pt depth -1.6pt width 23pt, {\em Stationary solutions
		of chemotaxis systems}, Trans. Amer. Math. Soc., 292 (1985), pp.~531--556.
	
	\bibitem{MR1752509}
	{\sc J.~Shatah and W.~Strauss}, {\em Spectral condition for instability}, in
	Nonlinear {PDE}'s, dynamics and continuum physics ({S}outh {H}adley, {MA},
	1998), vol.~255 of Contemp. Math., Amer. Math. Soc., Providence, RI, 2000,
	pp.~189--198.
	
	\bibitem{MR233240}
	{\sc S.~Steinberg}, {\em Meromorphic families of compact operators}, Arch.
	Rational Mech. Anal., 31 (1968/69), pp.~372--379.
	
	\bibitem{10780947202173109}
	{\sc I.~Takagi and C.~Zhang}, {\em Existence and stability of patterns in a
		reaction-diffusion-ode system with hysteresis in non-uniform media}, Discrete
	Contin. Dyn. Syst., 41 (2021), pp.~3109--3140.
	
	\bibitem{MR4213664}
	\leavevmode\vrule height 2pt depth -1.6pt width 23pt, {\em Pattern formation in
		a reaction-diffusion-{ODE} model with hysteresis in spatially heterogeneous
		environments}, J. Differential Equations, 280 (2021), pp.~928--966.
	
	\bibitem{MR3921221}
	{\sc J.~Wang}, {\em The stability of equilibria for a reaction-diffusion-{ODE}
		system on convex domains}, Appl. Math. Lett., 93 (2019), pp.~147--152.
	
	\bibitem{MR730252}
	{\sc H.~F. Weinberger}, {\em A simple system with a continuum of stable
		inhomogeneous steady states}, in Nonlinear partial differential equations in
	applied science ({T}okyo, 1982), vol.~81 of North-Holland Math. Stud.,
	North-Holland, Amsterdam, 1983, pp.~345--359.
	
	\bibitem{MR2573296}
	{\sc A.~Yagi}, {\em Abstract parabolic evolution equations and their
		applications}, Springer Monographs in Mathematics, Springer-Verlag, Berlin,
	2010.
	
\end{thebibliography}
\end{document}